\documentclass[11pt]{amsart}
\usepackage{amssymb,mathrsfs,bm,graphicx,geometry,mathtools,url} 
\usepackage[backref=page]{hyperref}
\usepackage[textsize=tiny]{todonotes}

\title{Statistics of divisor functions over function fields}
\author[M. Afifurrahman]{Muhammad Afifurrahman}
\thanks{
MA is supported by the National Research Foundation of Korea (NRF) grant funded by the Korea government (MSIT) (No. RS-2025-02262988).}

\address[MA]{Department of Mathematical Sciences, Ulsan National Institute of Science and Technology, 50 UNIST-gil, Ulsan 44919, South Korea}
\email{afifurrahman@unist.ac.kr}
\date{\today}

\def\cE{{\mathcal E}}
\def\cF{{\mathcal F}}

\def\cM{{\mathcal M}}

\def \F {{\mathbb F}}

\def \Z {{\mathbb Z}}

\def\energy{\mathsf{E}}
\def\denergy{\mathsf{D}}
\def\aenergy{\mathsf{A}}

\newtheorem{theorem}{Theorem}

\newtheorem{corollary}[theorem]{Corollary}

\newtheorem{proposition}[theorem]{Proposition}

\theoremstyle{remark}

\numberwithin{equation}{section}
\numberwithin{theorem}{section}
\numberwithin{table}{section}
\numberwithin{figure}{section}
\begin{document}

\begin{abstract}Extending recent works on the restricted divisor function over integers, we consider an analogue of this function over the function field. We provide some exact and asymptotic formulas related to this function.
\end{abstract}
\maketitle
\tableofcontents
\section{Introduction}
\subsection{Background and motivation over integers} 
The problem of computing statistics of the divisor function $\tau$ over the integers has a long history. Two particular problems of this kind that has arisen are estimating the \textit{second moments} and \textit{correlations} of $\tau$, respectively \begin{equation}\label{eqn:momcorZ}
\sum_{n=1}^H \tau(n)^2,\qquad \sum_{n=1}^H \tau(n)\tau(n+K),
\end{equation} for a fixed $K\ne 0$, as $H\to\infty$.

Elementary estimates imply that, as $H\to \infty$ \[
\sum_{n=1}^H \tau(n)^2\sim \dfrac{1}{\pi^2} H(\log H)^3 .
\] Meanwhile, a classic result of Ingham~\cite{Ing}, later improved by Estermann~\cite{Est}, implies \[
\sum_{n=1}^H \tau(n)\tau(n+K) \sim \dfrac{6}{\pi^2}\sigma^{-1}(K) H (\log H)^2
\] as $H\to \infty$, where $\sigma^{-1}(n)\coloneqq \sum_{d|n} d^{-1}$.

These problems are equivalent to counting integer solutions to \begin{equation}\label{eqn:integers}
x_1x_2-x_3x_4=K
\end{equation} with $0<x_1x_2,x_3x_4\le H$, as $H\to \infty$. Within this perspective, it is natural to instead count the number of integer solutions to~\eqref{eqn:integers} with $0<x_1,x_2,x_3,x_4\le H$. This modification naturally gives rise to the \textit{restricted divisor function} over integers $\tau_H$, defined as \[
\tau_H(n) \coloneqq \# \{ (a,b) \colon ab=n, 1\le a,b\le H\}.
\] Note that the support of $\tau_H$ is $\{1,\dots,H^2\}$. With respect to this function, counting solutions to~\eqref{eqn:integers} is equivalent to deriving an asymptotic formula for \begin{equation}\label{eqn:div-restricted}
    \sum_{i>0} \tau_H(i)\tau_H(i+K).
\end{equation} We first discuss the problem~\eqref{eqn:div-restricted} for the case $K=0$.
This set-up has had a long line of research and appears in various number theoretical settings, such as the problem of classifying maximal-size product sets of random subsets of $\Z$ (e.g.\:\cite{M}), determining the irrational numbers $\alpha$ for which the pair correlation for the fractional parts of $n^2\alpha$ is Poissonian (e.g.\:\cite{Truelsen,Shpa12I}), and counting $2\times 2$ singular integer matrices (as seen in~\cite{Afif4}).  In addition,~\eqref{eqn:integers} represents the multiplicative energy of the set $\{-H,\dots,H\}$.

An asymptotic formula for~\eqref{eqn:div-restricted} was first found by adapting a method of Ayyad--Cochrane--Zheng~\cite{ACZ}, subsequently improved by Mercier~\cite{Merc} and Shi~\cite[Theorem 1]{ShiS}, see also Liu--Zhai~\cite{LiuZhai}. The best result at the time of this work, due to Shi and Liu--Zhai, shows that the number of solutions to~\eqref{eqn:integers}, as $H\to \infty$ is  \begin{align*}
    \dfrac{12}{\pi^2} H^2\log H + A H^2  + O(H^{547/416+\varepsilon})
\end{align*} for any $\varepsilon>0$, with $$A=\frac{4\gamma-2\zeta'(2)/\zeta(2)-1-2\zeta(2)}{\zeta(2)}$$ and $\gamma$ being the Euler--Mascheroni constant.

When $K\ne 0$ in \eqref{eqn:div-restricted}, the problem is equivalent to estimating the shifted convolution of the restricted divisor function $\tau_H$. Another possible interpretation of this problem is counting $2\times 2$ (positive) integer matrices with a given determinant $\Delta$ whose entries are bounded above by $H$. In this setup, the problem was considered by the author at~\cite{Afif3} and further improved  in~\cite{CM25,CM26,DHP,GG,GGnew}. The best result as per the time of this preprint comes from~\cite{CM25}, which asserts that a fixed $\Delta\ne 0$ and as $H\to \infty$, the expression~\eqref{eqn:div-restricted} is asymptotically equal to \begin{align*}\dfrac{12}{\pi^2} \dfrac{\sigma(|\Delta|)}{|\Delta|} H^2+O(H^{1+o(1)}).
\end{align*}

Furthermore, when $n\ne 0$, the problem in~\eqref{eqn:div-restricted} can be further generalised to ~\eqref{eqn:integers} counting solutions to the system of equations \[
x_{1,1}\dots x_{1,m} =  \cdots = x_{k,1}\dots x_{k,m}
\] for integers $k,m\ge 2$, with $|x_{i,j}|\le H$ for all $i=1,\dots,k$, $j=1,\dots,m$. This setup can be  be interpreted as deriving asymptotics of the higher moments of some higher-dimensional restricted divisor functions. This family of equations was considered by Munsch--Shparlinski~\cite{MS} for the case $k=2$, Mastrostefano~\cite{M2} for $m=2$, and a recent work of the author and Corrigan~\cite{Afif4} for general $k,m\ge 2$.

\subsection{Motivation over function fields} 
We now provide some function field analogues of these divisor functions, which we define shortly. 
We first introduce some related notations. Let $\cM$  denote the set of all monic polynomials in $\mathbb{F}_q[T]$. Furthermore, let $\F_q[T]_{\le n}$ and $\F_q[T]_{n}$ be the sets of polynomials in $\F_q[T]$ of degree at most $n$ and exactly $n$, respectively. We define $\cM_{\le n}$  and $\cM_{n}$  analogously. Note, in particular, that the zero polynomial belongs to  $\F_q[T]_{\le n}$.

In the function fields setup, we consider the equation
 \begin{align}\label{eqn:multg}
f_1f_2-f_3f_4=g,
\end{align} over the ring of polynomials $\F_q[T]$, for a prime power $q$ and a fixed polynomial $g\in \F_q[T]$. In particular, we are interested in counting solutions to~\eqref{eqn:multg} where the degrees of the variables are restricted with respect to some parameters.

We consider the following two different models of counting solutions of~\eqref{eqn:multg}, as analogues of the integer cases discussed previously:
\begin{itemize}  \item $\deg f_1f_2,\deg f_3f_4 \le \delta$, (corresponds to the divisor function over integers $\tau$),  \item $\deg f_1,\deg f_2,\deg f_3,\deg f_4\le \delta$ (corresponds to the restricted divisor function over integers $\tau_H$).

\end{itemize}
There have been extensive studies in the literature on the problem of counting solutions to~\eqref{eqn:multg} with respect to the first model. This model is naturally related to the 
divisor function over $\F_q[T]$, defined as \[
\tau_k(f) \coloneqq \# \{(f_1,\dots,f_k) \in \cM^k \colon f=f_1\dots f_k \} \quad \text{for } f\in \cM.\] 
We  mostly consider the case $k=2$ in this work, where we drop the subscript $k$.

There have been many recent works surrounding the  statistics of $\tau_k(f)$, in particular for finding analogous results of~\eqref{eqn:momcorZ} over function fields.
For example, Andrade--Bary-Soroker--Rudnick~\cite{ABSR}  proved in~\cite[Theorem 1.1]{ABSR} the following asympotical formula that holds for a fixed $n$ and as $q\to \infty$, uniformly over nonzero $K \in \F_q[T]$ with $\deg K<n$: 
 \begin{equation} \label{eqn:tauk}
\sum_{f\in \cM_n} \tau_k(f)\tau_k(f+K) = \binom{n+k-1}{k-1}^2q^n + O(q^{n-1/2}) .
\end{equation}  
Furthermore, an exact formula for $k=2$ and $K\ne 0$ has been derived in various ways previously, see Conrey (unpublished), Florea--Lalin--Malik--Sahay~\cite{FLMS}, Gorodetsky~\cite{Gor}, Woo~\cite{Woo} and Yiasemides~\cite{Yia}. Note that the case $k=2$ and $K=0$ is immediate by considering the following standard generating function for $\tau$: \begin{align*}
    \sum_{f\in \cM} \tau(f)^2 u^{\deg f} = \dfrac{1-qu^2}{(1-qu)^4},
\end{align*}
however we provide an alternative proof in Theorem~\ref{thm:main1f}.

 Furthermore, such summation was also considered over short intervals, as seen in Keating--Rodgers--Roditty-Gershon--Rudnick~\cite{KRRR} and Kuperberg--Lalin~\cite{KL}. Yet another arithmetical sum problem over $\tau_k(f)$ was also studied in~\cite{ASS}. 

Our main goal of this paper is to obtain similar results over the second model. In particular, we aim to obtain similar result as in~\eqref{eqn:tauk} for the case $k=2$. As far as the author's knowledge at the time of this publication, this is the first such work on such setup. The related results will be discussed in the next section.

\subsection{Notation and conventions}
For a polynomial $g$, we often denote $\deg g=\delta_g$ as the degree of $g$. 
We use Landau big-O notation to describe the order of an expression. For $U$ and $V$ tending to infinity with respect to some variable(s) $x$, we use $U=O(V)$ if $|U|\le c V$ for some positive constant $c$ which may depend on the dimension variables (usually $m$ and $n$), which are fixed. Note that when $x$ represents some independent variables, we usually consider all variables tending to infinity, unless stated otherwise.
 We also use the notation
 $U\sim V$  if and only if $\lim\limits_{U,V\to \infty} U/V \to 1$. We also write $X^{o(1)}$ for an expression in $X$ whose exponent tends to 0 as $X \to \infty$.
\section{Main results}

\subsection{Results for $g=0$}
Returning to the equation~\eqref{eqn:multg},
we first discuss our results for the case $g=0$, where the problem becomes counting solutions over $\F_q[T]$ to \begin{equation}\label{eqn:multen}
    f_1f_2=f_3f_4.
\end{equation}

For the restricted divisor model, we consider the problem of deriving formulas for the following quantities: \begin{align*}
&\energy(\F_q[T]_{\le\delta}) \coloneqq \#\{ (f_1,f_2,f_3,f_4)\in \F_q[T]_{\le \delta}^4 \colon f_1f_2=f_3f_4\}, \\
&\energy( \cM_{\le \delta} ) \coloneqq \#\{ (f_1,f_2,f_3,f_4)\in \cM_{\le \delta}^4 \colon f_1f_2=f_3f_4\},
\end{align*}  for integer $\delta\ge 0$.
Note that for $\energy(\F_q[T]_{\le\delta}) $, we also include constant polynomials (of degree $0$) and zero polynomial (defined as of degree $-\infty$).

We first provide a heuristic for the expected order of $\energy(\F_q[T]_{\le\delta})$. Fix the polynomials $f_1,f_2$, which can be done in $(q^{\delta+1})^2=q^{2\delta+2}$ ways. For each such choice, we note that the tuples $(f_1,f_2,\alpha f_1,\alpha^{-1}f_2)$ is a solution of~\eqref{eqn:multen}, for any nonzero $\alpha \in \F_q$. Since there are $q-1$ choices of $\alpha$, we expect the quantity $\energy(\F_q[T]_{\le\delta})$ to be in the order of $q^{2\delta+3}$.

We confirm this heuristic in the following result.

\begin{theorem}\label{thm:maininf}
    We have \begin{align*}
    &\energy(\F_q[T]_{\le\delta})= (\delta+1)q^{2\delta+3}+q^{2\delta+2}-(\delta+1)q^{2\delta+1},\\
   & \energy( \cM_{\le \delta} ) = \dfrac{\energy(\F_q[T]_{\le\delta}) - (2q^{\delta+1}-1)^2}{(q-1)^3}.
    \end{align*}
\end{theorem}

In a related setup, we also consider the problem of counting solutions to~\eqref{eqn:multen} where we fix the degree of each polynomials. For integers $\delta_1,\delta_2,\delta_3,\delta_4\ge 0$, we let  \begin{align*}
&\energy^{\cF} (\F_q[T];\delta_1,\delta_2,\delta_3,\delta_4) \coloneqq \# \{ (f_1,f_2,f_3,f_4)\in \F_q[T]^4 \colon f_1f_2=f_3f_4, f_i \in \F_q[T]_{\delta_i} \text{ for }i=1,2,3,4\},\\
&\energy^{\cF}  (\cM;\delta_1,\delta_2,\delta_3,\delta_4) \coloneqq \# \{ (f_1,f_2,f_3,f_4)\in \cM^4 \colon f_1f_2=f_3f_4, f_i \in \cM_{\delta_i} \text{ for }i=1,2,3,4\}.
\end{align*} where $\delta_1,\delta_2,\delta_3,\delta_4\ge 0$. 
Indeed, in order for the above expressions to be nonzero, we may let $\delta_1+\delta_2=\delta_3+\delta_4=\delta$. The following result provides a formula for these quantities.

\begin{theorem}\label{thm:mainfixed}
Let $\delta_1,\delta_2,\delta_3,\delta_4\ge 0$ be integers with $\delta_1+\delta_2=\delta_3+\delta_4=\delta$. Define $\min(\delta_1,\delta_2,\delta_3,\delta_4)=\delta_{\min}$. Then, \begin{align*}
&\energy^{\cF}  (\cM;\delta_1,\delta_2,\delta_3,\delta_4)= \begin{cases}1, & \text{when } \delta=0,\\
      (\delta_{\min}+1)q^{\delta}-\delta_{\min}q^{\delta-1}, & \text{when } \delta>0,
  \end{cases}\\
&  \energy^{\cF} (\F_q[T];\delta_1,\delta_2,\delta_3,\delta_4) = (q-1)^3\energy^{\cF}  (\cM;\delta_1,\delta_2,\delta_3,\delta_4).
\end{align*}
\end{theorem}

In addition, our arguments can also be applied to derive similar results over the classical divisor function model in $\F_q[T]$. We let
\begin{align*}
&\denergy^{f} (\F_q[T]_{\delta}) \coloneqq \#  \cE^{f,1} \{ (f_1,f_2,f_3,f_4)\in \F_q[T] \colon f_1f_2=f_3f_4 \in \F_q[T]_\delta\},\\
&\denergy^{f} (\cM_{\delta}) \coloneqq \#  \cE^{f,1} \{ (f_1,f_2,f_3,f_4)\in \cM \colon f_1f_2=f_3f_4 \in \cM_\delta\},
\end{align*}
and recover the following formulas for these quantities.
\begin{theorem}\label{thm:main1f}
    We have \begin{align*}
    &\denergy^{f} (\cM_{\delta}) = \begin{cases} 1,  & \text{when } \delta=0,\\
       \displaystyle  \binom{\delta+3}{3} q^{\delta} - \binom{\delta+1}{3}q^{\delta-1}, & \text{when } \delta\ge 1,
    \end{cases}\\
    &\denergy^{f} (\F_q[T]_{\delta}) = (q-1)^3 \denergy^{f} (\cM_{\delta}).
    \end{align*}
\end{theorem}

Akin to Theorem~\ref{thm:maininf}, we also consider a bounded version of the previous model, by considering \begin{align*}
&\denergy^{f} (\F_q[T]_{\le\delta}) \coloneqq \#  \cE^{f,1} \{ (f_1,f_2,f_3,f_4)\in \F_q[T] \colon f_1f_2=f_3f_4 \ne 0 \in \F_q[T]_{\le\delta}\},\\
&\denergy^{f} (\cM_{{\le\delta}}) \coloneqq \#  \cE^{f,1} \{ (f_1,f_2,f_3,f_4)\in \cM \colon f_1f_2=f_3f_4 \in \cM_{\le\delta}\}.
\end{align*} Note that in this case, we exclude the cases of zero polynomials for $\F_q[T]$. 
For the sake of completeness, we state the following formulas for these quantities as a corollary to Theorem~\ref{thm:main1f}.
\begin{corollary}\label{cor:main1}
  We have \begin{align*}
    &\denergy^{f} (\cM_{\le\delta}) =\binom{\delta+3}{3}q^{\delta}  + \sum_{i=0}^{\delta-1}\binom{i+2}{2}q^i,\\
    &\denergy^{f} (\F_q[T]_{\le\delta}) = (q-1)^3 \denergy^{f} (\cM_{\le \delta}).
    \end{align*}
\end{corollary}

\subsection{Results for $g\ne 0$}We now consider the problem of counting the number of solutions to \eqref{eqn:multg} for a fixed $g\ne 0$. We first discuss the restricted divisor function analogue, where we consider~\eqref{eqn:multg} with $\deg f_i\le \delta$, $\deg g=\delta_g<2\delta$. 
We let \begin{align*}
&\aenergy( \cM_{\le \delta} ;g) \coloneqq \#\{ (f_1,f_2,f_3,f_4)\in \cM_{\le \delta}^4\colon f_1f_2-f_3f_4=g \}.
\end{align*} 
Note that we only consider monic polynomial solutions in this part, since the arguments to obtain non-monic solutions will be similar. We now provide the following asymptotic formula. \begin{theorem}\label{thm:detgen}
  Let $g\in \F_q[T]$ be a nonzero nonconstant polynomial. Fix an integer $\delta>\delta_g$. Then, as $q \to \infty$, we have \[
    \aenergy(\cM_{\le\delta};g)= q^{2\delta}+O(q^{2\delta-1}).
    \]
\end{theorem}
However, when $g=\Delta \in \F_q^*$, we instead obtain the following exact formula for $\aenergy(\F_q[T]_{\le\delta};\Delta)$.
\begin{theorem}\label{thm:detconstant}
    Let $g=\Delta\in \F_q^*$. We have \begin{align*}
        \aenergy(\cM_{\le\delta};\Delta)= 
         \displaystyle   \dfrac{q(q^{\delta}-1)(q^{\delta+1}+q^{\delta}+q-3)}{(q-1)^2}.
    \end{align*} 
\end{theorem}

This result serves as an analogue for~\cite[Theorem 7.1]{ABSR} in our model.

As an addition, we also discuss the application of our arguments for the divisor function model, which is to count solutions to~\eqref{eqn:multg} with $\deg(f_1f_2)=\deg(g_1g_2)=\delta$. We let \begin{align*}
&\aenergy^f( \cM_{\le \delta} ;g) \coloneqq \#\{ (f_1,f_2,f_3,f_4)\in \cM^4\colon f_1f_2-f_3f_4=g, f_1f_2,f_3f_4\in \cM_{\le \delta} \},
\end{align*} and first note that this quantity is exactly $\sum \tau(f)\tau(f+K)$ in~\eqref{eqn:tauk}. By applying our argument in this setup, then obtain the following asymptotical result, which recovers the main term of~\cite{ABSR,FLMS,Gor,Woo,Yia}.

\begin{theorem}\label{thm:detgendiv}
      Let $g\in \F_q[T]$ be a nonzero polynomial and fix an integer $\delta>\delta_g$. Then, as $q\to \infty$,  we have \[
    \aenergy^{f}(\cM_{\le\delta};g)= (\delta+1)^2 q^{\delta}+O(q^{\delta-1}).\]
\end{theorem}

We now describe the outline of our arguments in this paper, which are mostly elementary. We start with reducing the counting to monic polynomials, as seen in Section~\ref{sec:tomonic} for example. Then, we use a version of Euler's four-number theorem for $\F_q[T]$, as seen in Proposition~\ref{prp:fournumber}
We then count the number of such tuples over all possible pairs of relatively prime polynomials $(g_2,h_1)$, with additional restrictions imposed by each of the theorems, and sum accordingly. 

\section{Preliminary results}
The following proposition is a function field analogue of Euler's four-number theorem over $\Z$, which we use extensively in our proof.\begin{proposition}\label{prp:fournumber}
    Let $f_1,f_2,f_3,f_4 \in \F_q[T]$ be monic polynomials satisfying~\eqref{eqn:multen}. Then, there uniquely exists monic polynomials $g_1,g_2,h_1,h_2\in \F_q[T]$ such that \begin{equation*}
f_1=g_1g_2, \quad f_2=h_1h_2, \quad f_3=g_1h_1, \quad f_4=g_2h_2
\end{equation*} and $\gcd(g_2,h_1)=1$.
\end{proposition}

Another result used in our argument is counting pairs of relatively prime polynomials with prescribed degrees. These results can be derived by applying a result of Benjamin and Bennett~\cite{BB}, see also Hou and Mullen~\cite{HM}, on the probabilites of having such pairs. 

For non-negative integers $m,n$, denote $F(m,n)$ as the number of pairs of relatively prime monic polynomials $(f,g)\in \F_q[T]$ with $\deg f=m$, $\deg g=n$. The following result is derived by a direct application of~\cite[Theorem~3]{BB}.\begin{proposition}\label{prp:relprim}
 We have \[ F(m,n)=
\begin{cases}
    1, &m=n=0,\\
    q^n, &m=0,n>0,\\
    q^m, &m>0,n=0,\\
    (q-1)q^{m+n-1}, &m,n>0.
\end{cases}
    \]
    \end{proposition}

Other  results we use in this work is related to counting solutions to~\eqref{eqn:multg} for fixed $f_2,f_4$, as stated in the following propositions, taken as generalisations  of~\cite[Lemmas 7.2 and 7.3]{ABSR}. The proofs and the idea are similar as the proofs in~\cite{ABSR}, only that we replace $1$ with $g$ as needed.
\begin{proposition}\label{prp:72} Let $h_2$ and $h_4$ be fixed monic polynomials in $\F_q[T]$ with $\gcd(h_2,h_4)=1$. Then, the number of solutions $(f_1,f_2,f_3,f_4)\in \cM^4$ to~\eqref{eqn:multg} with $(f_2,f_4)=(h_2,h_4)$ and \[
\deg f_1+\deg h_2 = \deg f_3+\deg h_4=n,
\] where $n$ is a positive integer with $n\ge \deg h_2+\deg h_4$ and $n>\deg g$, is $q^{n-\deg h_2-\deg h_4}$.
When $n\le \deg g$, no such solution exists.
\end{proposition}

\begin{proof}[Proof of Proposition~\ref{prp:72}]This argument follows the line of~\cite[Lemma 7.2]{ABSR}, to which we reproduce for the sake of completeness.
First note that since $f_1h_2$ and $f_3h_4$ are monic polynomials of degree $n$, the leading terms of these polynomials would cancel out when evaluating $f_1h_2-f_3h_4$. Hence, if such pair exists, we need to have \[\deg(f_1h_2-f_3h_4)=\deg g <n.\]

    We then prove that the set of pairs of polynomials $(f_1,f_3)$ with $\deg f_1+\deg h_2 = \deg f_3+\deg h_4=n$ and \begin{equation}\label{eqn:fhg}
    f_1h_2-f_3h_4=g
    \end{equation}  forms a non-empty affine space of dimension $n-\deg h_2-\deg h_4$. We first ignore the degree condition, and note that the above equation trivially has a solution $(u_0,v_0)$ by applying Bezout's theorem over the ring $\F_q[T]$. Furthermore, all solutions to~\eqref{eqn:fhg} can be parametrised as \[
    (f_1,f_3)=(u_0,v_0)+k(h_4,h_2), \quad k\in \F_q[T].
    \]
    
    From this observation, we may replace $u_0$ with $u_1=u_0+kh_4$ such that $\deg u_1<\deg h_4$ (or zero); hence, we may assume $\deg u_0<\deg h_4$. Then, we may imply that if $k\ne 0$, $\deg (u_0+kh_4)=\deg(kh_4)$. Furthermore, $u_0+kh_4$ is monic if and only if $k$ is also monic. Hence, if $k\ne 0$, then \begin{align*}
        \deg(u_0+kh_4)=n-\deg h_2 \iff \deg kh_4=n-\deg h_2 \iff \deg k = n-\deg h_2-\deg h_4.
    \end{align*}
    This implies a one-to-one correspondence between the set of solutions $(f_1,f_3)$ to~\eqref{eqn:fhg} with $\deg f_1+\deg h_2 = \deg f_3+\deg h_4=n,$ and the set of monic polynomials of degree exactly $n-\deg h_2-\deg h_4$. 
\end{proof}

\begin{proposition}\label{prp:73} Let $n>\deg g$ be a positive integer. The number of solutions $(f_1,f_2,f_3,f_4)\in \cM^4$ to~\eqref{eqn:multg} with  $\deg f_i=\delta_i$ for $i=1,2,3,4$, $\delta_1+\delta_2=\delta_3+\delta_4=n$ and $\gcd(f_2,f_4)=1$ is exactly  \[
   \begin{cases}
        0, &\quad \text{if }n=1,\\
        q^n, &\quad \text{if } \min(\delta_1,\delta_2,\delta_3,\delta_4)=0,\\
        q^n-q^{n-1}, &\quad \text{else.}
   \end{cases}
    \]
\end{proposition}

Note that Proposition~\ref{prp:73} is uniform - that is, except for the choice of $n$, the result does not depend on the choice of $g$. Furthermore, the restriction $n\ge \deg g$ is natural from~\eqref{eqn:multg}.

\begin{proof}[Proof of Proposition~\ref{prp:73}]
    Without loss of generality, as in the proof for~\cite[Lemma 7.3]{ABSR}, we may assume $\delta_1\ge \delta_2$, $\delta_3\ge \delta_4$. Then, the number of solutions to~\eqref{eqn:multg} in this setup is exactly \begin{equation}\label{eqn:restricted}
\sum_{\substack{{h_2\in \cM_{\delta_2},h_4\in \cM_{\delta_4}}\\
\gcd(h_2,h_4)=1}} \# \{ f_1\in \cM_{\delta_1} , f_3\in \cM_{\delta_3} \colon f_1h_2-f_3h_4=g \}.
    \end{equation}
Since $\delta_1\ge \delta_2$, $\delta_3\ge \delta_4$ and $\delta_1+\delta_2=\delta_3+\delta_4=n$, we have $\delta_2+\delta_4 \le n$. Therefore, from Proposition~\ref{prp:72}, we have that for a fixed $(h_2,h_4)$, \[
\# \{ f_1\in \cM_{\delta_1} , f_3\in \cM_{\delta_3} \colon f_1h_2-f_3h_4=g \} = q^{n-\deg h_2-\deg h_4},
\]
which implies \begin{align*}
    \sum_{\substack{{h_2\in \cM_{\delta_2},h_4\in \cM_{\delta_4}}\\
\gcd(h_2,h_4)=1}} \# \{ f_1\in \cM_{\delta_1} , f_3\in \cM_{\delta_3} \colon f_1h_2-f_3h_4=g \} &=  \sum_{\substack{{h_2\in \cM_{\delta_2},h_4\in \cM_{\delta_4}}\\
\gcd(h_2,h_4)=1}} q^{n-\delta_2-\delta_4} \\
&= q^{n-\delta_2-\delta_4} \sum_{\substack{{h_2\in \cM_{\delta_2},h_4\in \cM_{\delta_4}}\\
\gcd(h_2,h_4)=1}} 1.
\end{align*}
It remains to compute the last term, which is related to the number of relatively prime polynomial pairs $(h_2,h_4)$ of prescribed degree. Using notations from Proposition~\ref{prp:relprim}, we then have
\begin{equation*}
    q^{n-\delta_2-\delta_4} \sum_{\substack{{h_2\in \cM_{\delta_2},h_4\in \cM_{\delta_4}}\\
\gcd(h_2,h_4)=1}} 1 = q^{n-\delta_2-\delta_4} F(\delta_2,\delta_4) = \begin{cases}
    q^{n}, &\text{ when} \min(\delta_2,\delta_4)=0,\\
    q^n-q^{n-1}, &\text{ else.}
\end{cases}
\end{equation*}

\end{proof}

We need separate counting results when $n=\deg g>0$, which does not happen in~\cite{ABSR}. In this case, any solutions to~\eqref{eqn:multg} will have $\deg(f_1f_2)=n>\deg(f_3f_4)$, except when $\F_q$ is of characteristic $2$.

\begin{proposition}\label{prp:72g} Suppose $n=\deg g>0$. Let $h_2$ and $h_4$ be fixed monic polynomials in $\F_q[T]$ with $\gcd(h_2,h_4)=1$ and $n \ge \deg h_2+\deg h_4$. Then, the number of solutions $(f_1,f_2,f_3,f_4)\in \cM^4$ to~\eqref{eqn:multg} with $(f_2,f_4)=(h_2,h_4)$ and $\deg (f_1h_2)=g$ is \[
\dfrac{q^{n-\deg h_2-\deg h_4}-1}{q-1}+\varepsilon,
\] where $\varepsilon$ is $1$ if and only if $-gh_4^{-1} \pmod {h_2}$ is a monic polynomial, and $0$ else.
\end{proposition}

\begin{proof} We first note that $\deg f_3\le n - \deg h_4-1$. Then, we have \[
    f_1h_2-f_3h_4=g \iff f_1 = \dfrac{f_3h_4+g}{h_2}.
    \] Hence, the problem is equivalent to counting the number of monic polynomials $f_3$ with $f_3h_4\equiv -g \pmod {h_2}$. 
    Such polynomial is of the form \[
    f_3=h_2r+s,\qquad r\in \F_q[T],
    \] where $s$ is the unique polynomial with $\deg s <\deg h_2$ and $h_2 \mid h_4s+g$. Additionally, either $r$ must be monic or $r=0$ and $s$ is monic.

    We now count the number of monic polynomials of such form with $\deg f_3=e$. Trivially, we have $\deg h_2 \le e \le  n- \deg h_4-1$, except when $r=0$.
For each possible degree $e$ in the implied interval, there are $q^{e-\deg a}$ number of monic choices of $f_3$. In addition, when $e=\deg h_2$, there are $1$ or $0$ possible choices depending on whether $s$ is monic. Summing over all possible $e$ completes our proof.
\end{proof}

Applying similar arguments as in the proof of Proposition~\ref{prp:73}, we then obtain the following result, in asymptotical form.

\begin{proposition}\label{prp:73g} Fix $\deg g=n$, and nonnegative integers $\delta_1,\delta_2,\delta_3,\delta_4$ with $\delta_1+\delta_2=n>\delta_3+\delta_4$. Then, as $q\to \infty$, there are at most $O(q^{n-1})$ number of solutions $(f_1,f_2,f_3,f_4)\in \cM^4$ to~\eqref{eqn:multg} with   $\gcd(f_2,f_4)=1$ and $\deg f_i=\delta_i$ for $i=1,2,3,4$.
\end{proposition}

\begin{proof}
    Following the proof of Proposition~\ref{prp:73}, we first develop an explicit sum for the number of such tuples with $\deg f_i=\delta_i$ for $i=1,2,3,4$ and $\gcd(f_2,f_4)=1$. In this setup, (with letting $\delta_1\ge \delta_2$ and $\delta_3\ge \delta_4$, so that $n>\delta_2+\delta_4$), the number of such tuples is exactly
    \begin{align*}
&\sum_{\substack{{h_2\in \cM_{\delta_2},h_4\in \cM_{\delta_4}}\\
\gcd(h_2,h_4)=1}} \# \{ f_1\in \cM_{\delta_1} , f_3\in \cM_{\delta_3} \colon f_1h_2-f_3h_4=g \} \\ &\quad= \sum_{\substack{{h_2\in \cM_{\delta_2},h_4\in \cM_{\delta_4}}\\
\gcd(h_2,h_4)=1}} \left[\dfrac{q^{n-\delta_2-\delta_4}-1}{q-1} +\varepsilon(h_2,h_4) \right]
\\
&\quad = \sum_{\substack{{h_2\in \cM_{\delta_2},h_4\in \cM_{\delta_4}}\\
\gcd(h_2,h_4)=1}} \varepsilon(h_2,h_4) +\begin{cases}
    \dfrac{q^n-q^{\delta_2+\delta_4}}{q-1}, &\text{when } \min(\delta_2,\delta_4)=0,\\
    q^{n-1}-q^{\delta_2+\delta_4-1}, &\text{else,}
\end{cases} 
    \end{align*} with $\varepsilon(h_2,h_4)=1$ if and only if $-gh_4^{-1} \pmod{h_2}$ is monic (and $0$ else). Taking the asymptotic completes the proof. 
\end{proof}

\section{Proof of Theorem~\ref{thm:maininf}}
\subsection{Reducing the caseworks}\label{sec:tomonic}
To count solutions to~\eqref{eqn:multen}, we reduce our counting to the case where all polynomials are monic. We let
\begin{align*}
\energy_0 (\F_q[T];\delta) \coloneqq \#\{ (f_1,f_2,f_3,f_4)\in \cE (\F_q[T];\delta) \colon 0\in \{f_1,f_2,f_3,f_4\}\},
\end{align*}
and thus we are ready to prove the following proposition.
\begin{proposition}\label{prp:tomonic}
    We have \[
    \energy(\F_q[T]_{\le\delta})= (2q^{\delta+1}-1)^2 + (q-1)^3  \energy(\cM_{\le\delta}).
   \]
\end{proposition}

To prove this proposition, we first show
\begin{equation}\label{eqn:nonzero}
     \energy(\F_q[T]_{\le\delta})-\energy_0 (\F_q[T];\delta) = (q-1)^3  \energy(\cM_{\le\delta}).
\end{equation}

To prove~\eqref{eqn:nonzero}, we first note that  $\energy(\F_q[T]_{\le\delta})-\energy_0 (\F_q[T];\delta)$ counts the number of solutions of~\eqref{eqn:multen} such that none of the polynomials are zero. Consider such solution $(f_1,f_2,f_3,f_4)$ and let the leading coefficients of the polynomials be $\alpha_1$, $\alpha_2$, $\alpha_3$ and $\alpha_4$ respectively. By comparing the leading coefficient of the polynomial $f_1f_2-f_3f_4$, we obtain \begin{align}\label{eqn:alpha}
        \alpha_1\alpha_2=\alpha_3\alpha_4.
    \end{align} Furthermore, we note that \[
    (f_1/\alpha_1,f_2/\alpha_2,f_3/\alpha_3,f_4/\alpha_4)
    \] is a solution of~\eqref{eqn:multen} such that all corresponding polynomials are monic.

    From this argument, we see that each monic solution of~\eqref{eqn:multen} would contribute to $K$ general solutions of~\eqref{eqn:multen}, where $K$ is the number of solutions to~\eqref{eqn:alpha} over $\F_q^*$. We observe that a fixed nonzero choice of $(\alpha_1,\alpha_2,\alpha_3)$ results in a unique value for $\alpha_4$. Therefore, there are exactly $(q-1)^3$ solutions to~\eqref{eqn:alpha}, which therefore completes the proof of~\eqref{eqn:nonzero}.

Next, we prove \begin{equation}\label{eqn:zero}
    \energy_0 (\F_q[T];\delta)=(2q^{\delta+1}-1)^2.
\end{equation}
    Consider a solution of~\eqref{eqn:multen} such that at least one of the variables is equal to zero. In fact, we may let at least two of them are zero polynomials. We then divide the caseworks based on the number of nonzero polynomials as follows:
    \begin{itemize}
        \item When all polynomials are zero, we trivially have $1$ solutions;
        \item When exactly three of the polynomials are zero, we have $4$ choices for the nonzero variable, and $q^{\delta+1}-1$ choices for the nonzero polynomial. We then obtain $4(q^{\delta+1}-1)$ solutions in this case.
        \item When exactly two polynomials are zero, we first have $4$ choices of the zero variables (since the pairs $(f_1,f_2)$ and $(f_3,f_4)$ can not be both equal to zero in this case). Next, we have $(q^{\delta+1}-1)^2$ choices for the nonzero polynomials, since each choice is independent. Therefore, we obtain $4(q^{\delta+1}-1)^2$ solutions in this case.
    \end{itemize}
    Adding all caseworks, we then obtain \[
      \energy_0 (\F_q[T];\delta) = 4(q^{\delta+1}-1)^2+4(q^{\delta+1}-1) +1 = (2(q^{\delta+1}-1)+1)^2= (2q^{\delta+1}-1)^2,
    \]
which completes the proof of~\eqref{eqn:zero}. The proof of Proposition~\ref{prp:tomonic} is now complete by substituting~\eqref{eqn:zero} to~\eqref{eqn:nonzero}.

\subsection{Counting monic polynomials} Based on Proposition~\ref{prp:tomonic}, it now remains to count $\energy(\cM_{\le\delta})$, the number of monic polynomial tuples $(f_1,f_2,f_3,f_4) \in \F_q[T]^4$ with \[
f_1f_2=f_3f_4 
\] such that $0\le \deg f_1,\deg f_2,\deg f_3,\deg f_4\le \delta$.
By applying Proposition~\ref{prp:fournumber}, our problem is equivalent to counting tuples of monic polynomials \[
(g_1,g_2,h_1,h_2)
\] with $\gcd(g_2,h_1)=1$ and \[
\deg g_1, \deg h_2 \le \delta-\max(\deg g_2,\deg h_1).
\]

Fix a choice of $g_2$ and $h_1$.  Then, by considering the number of choices for each of $g_1$ and $h_2$, we obtain \begin{align}\label{eqn:sumgcd} \begin{split}
\energy(\cM_{\le\delta}) &=\sum_{\gcd(g_2,h_1)=1} (1+q+\dots+q^{\delta-\max(\deg g_2,\deg h_1)})^2\\&=\sum_{\gcd(g_2,h_1)=1} \left(\dfrac{q^{\delta+1-\max(\deg g_2,\deg h_1)}-1}{q-1}    \right)^2.
\end{split}
\end{align}

By applying this result over all possible pairs $(g_2,h_1)$ with fixed degrees $\deg g_2=m$, $\deg h_1=m$, we expand the summation in~\eqref{eqn:sumgcd} as follows: \begin{align*}
\energy(\cM_{\le\delta}) &=\sum_{m,n=0}^\delta F(m,n)\left(\dfrac{q^{\delta+1-\max(m,n)}-1}{q-1}    \right)^2.
\\
&=\left(\dfrac{q^{\delta+1}-1}{q-1}    \right)^2 + 2 \sum_{i=1}^\delta q^i \left(\dfrac{q^{\delta+1-i}-1}{q-1}    \right)^2 + \sum_{m,n=1}^\delta \dfrac{q^{m+n-1}(q^{\delta+1-\max{(m,n)}}-1)^2}{(q-1)^2}.
\end{align*}

Substituting this to Proposition~\ref{prp:tomonic}, we obtain

\begin{align}\label{eqn:Efirst}\begin{split}
  \energy(\F_q[T]_{\le\delta})   &= (2q^{\delta+1}-1)^2 +
(q-1)\left({q^{\delta+1}-1}    \right)^2 + 2 (q-1)\sum_{i=1}^\delta q^i \left({q^{\delta+1-i}-1}    \right)^2 \\ & \qquad + (q-1)^2\sum_{m,n=1}^\delta {q^{m+n-1}(q^{\delta+1-\max{(m,n)}}-1)^2}
\\ &= 4q^{2\delta+2} -4q^{\delta+1}+1 +  (q-1)(q^{2\delta+2}-2q^{\delta+1}+1) +  2 (q-1)\sum_{i=1}^\delta q^i \left({q^{\delta+1-i}-1}    \right)^2 \\
& \qquad + (q-1)^2 \sum_{i=1}^\delta (q^{2i-1}+2(q^i+\dots+q^{2i-2}))(q^{\delta+1-i}-1)^2\\
\\ &=  q^{2\delta+3} +3q^{2\delta+2} -2q^{\delta+2} -2q^{\delta+1} +q +  2 (q-1)\sum_{i=1}^\delta q^i \left({q^{\delta+1-i}-1}    \right)^2 \\
& \qquad + (q-1)\sum_{i=1}^\delta (q^{2i} + q^{2i-1} - 2q^i)(q^{\delta+1-i}-1)^2
\\ & =  q^{2\delta+3} +3q^{2\delta+2} -2q^{\delta+2} -2q^{\delta+1} +q + (q-1)\sum_{i=1}^\delta (q^{2i} + q^{2i-1} )(q^{\delta+1-i}-1)^2. 
\end{split}
\end{align}

We now simplify the last term of the last line of~\eqref{eqn:Efirst} as follows: \begin{align*}
    (q-1)\sum_{i=1}^\delta (q^{2i} + q^{2i-1} )(q^{\delta+1-i}-1)^2 & = (q-1)(q+1)\sum_{i=1}^\delta q^{2i-1}(q^{2\delta+2-2i} -2q^{\delta+1-i}+1)
    \\ &= (q^2-1)\sum_{i=1}^\delta (q^{2\delta+1} -2q^{\delta+i} + q^{2i-1})
    \\ &= (q^2-1)\left(\delta q^{2\delta+1}- 2q^{\delta+1}\cdot \dfrac{q^{\delta}-1}{q-1} + q \cdot\dfrac{q^{2\delta}-1}{q^2-1}  \right)
    \\ &= \delta q^{2\delta+1}(q^2-1) -2q^{\delta+1}(q^\delta-1)(q+1) + q(q^{2\delta}-1)
    \\ & = \delta q^{2\delta+3} -2q^{2\delta+2}  -(\delta+1)q^{2\delta+1} +2q^{\delta+2}+2q^{\delta+1}-q.
\end{align*}

Substituting this to~\eqref{eqn:Efirst}, we obtain \begin{align*}
   \energy(\F_q[T]_{\le\delta})   &= q^{2\delta+3} +3q^{2\delta+2} -2q^{\delta+2} -2q^{\delta+1} +q  +   \delta q^{2\delta+3} -2q^{2\delta+2}  -(\delta+1)q^{2\delta+1} +2q^{\delta+2}+2q^{\delta+1}-q  \\
   & = (\delta+1)q^{2\delta+3} +q^{2\delta+2} -(\delta+1)q^{2\delta+1},
\end{align*} which completes our proof.

\section{Proof of Theorem~\ref{thm:mainfixed}} 
By considering permutations between the variables, we may let $\delta_1\ge \delta_3  \ge \delta_4 \ge  \delta_2$, which implies $\delta_{\min}=\delta_2$.
We then transform our problem to counting monic polynomials, by applying similar arguments as in Section~\ref{sec:tomonic}, In particular, similar to Proposition~\ref{prp:tomonic},  we have \begin{equation}
    \energy^{\cF}(\F_q[T];\delta_1,\delta_2,\delta_3,\delta_4) =(q-1)^3 \energy^{\cF}  (\cM;\delta_1,\delta_2,\delta_3,\delta_4),
\end{equation} with \[
\energy^{\cF}  (\cM;\delta_1,\delta_2,\delta_3,\delta_4) \coloneqq \#\{ (f_1,f_2,f_3,f_4)\in \energy^{\cF}(\F_q[T];\delta_1,\delta_2,\delta_3,\delta_4)   \colon f_1,f_2,f_3,f_4\in \cM\}.
\]
The term differences are due to the fact that our restrictions in this case already excluded zero polynomials.

Again, by applying Proposition~\ref{prp:fournumber},  our problem is equivalent to counting tuples of monic polynomials $
(g_1,g_2,h_1,h_2)$ with $\gcd(g_2,h_1)=1$ and \[
\deg g_1+\deg g_2 = \delta_1, \quad \deg h_1+\deg h_2=\delta_2, \quad \deg g_1+\deg h_1=\delta_3, \quad \deg g_2+\deg h_2 =\delta_4.
\]
Note that in the case $\delta=0$, we trivially have $\energy^{\cF}  (\cM;\delta_1,\delta_2,\delta_3,\delta_4)=1$. Therefore, the proof is complete in this case, and we may now let $\delta>0$.

Recall that we assume $\delta_1\ge \delta_3\ge \delta_4 \ge \delta_2$.
Note that for a fixed choice $g_2$ and $h_1$, $\deg g_1$ and $\deg h_2$ is uniquely determined. Furthermore, such choice is only possible if and only if \[\deg g_2 -\deg h_1 =\delta_4-\delta_2 \coloneqq \delta'\]
Note that from our assumption, $\deg g_2\ge \deg h_1$.
Hence, for each relatively prime pairs of $(g_2,h_1)$ with $\deg g_2-\deg h_1=\delta'$, there are  \[q^{\delta_1-\deg g_2}q^{\delta_2-\deg h_1 } = q^{\delta-\delta'-2\deg h_1}\] choices for $(g_1,h_2)$, where we recall $\delta_1+\delta_2=\delta$.

Then, we have  \begin{equation}\label{eqn:efixedmonic}
    \energy^{\cF}  (\cM;\delta_1,\delta_2,\delta_3,\delta_4) = \sum_{\substack{
       \gcd(g_2,h_1)=1 \\ \deg g_2 \le \delta_4\\ \deg h_1 \le \delta_2 \\ \deg g_2-\deg h_1=\delta'}
    } q^{\delta_1-\deg g_2}q^{\delta_2-\deg h_1}  = \sum_{\substack{\gcd(g_2,h_1)=1 \\ 0\le \deg h_1\le \delta_2\\ \deg g_2=\deg h_1+\delta'}} q^{\delta-\delta'-2\deg h_1}.
\end{equation}
Applying Proposition~\ref{prp:relprim} with respect to the last expression, we obtain 
\begin{equation}
    \energy^{\cF}(\F_q[T];\delta_1,\delta_2,\delta_3,\delta_4) = \sum_{i=0}^{\delta_2}  F(i+\delta',i)  q^{\delta-\delta'-2i}.
\end{equation}

First, when $\delta'=0$, we obtain
\begin{align*}
    \energy^{\cF}  (\cM;\delta_1,\delta_2,\delta_3,\delta_4) &= \sum_{i=0}^{\delta_2}  F(i,i)  q^{\delta-2i}= q^{\delta} + \sum_{i=1}^{\delta_2} (q-1)q^{\delta-1} = (\delta_2+1)q^{\delta} - \delta_2q^{\delta-1}.
\end{align*}
Then, when $\delta'\ne 0$, we obtain
\begin{align*}
           \energy^{\cF}(\F_q[T];\delta_1,\delta_2,\delta_3,\delta_4)  &= \sum_{i=0}^{\delta_2}  F(i+\delta',i)  q^{\delta-\delta'-2i}
    \\& = q^{\delta'} \cdot q^{\delta-\delta'} + \sum_{i=1}^{\delta_2 } (q-1)q^{2i+\delta'-1} \cdot q^{\delta-\delta'-2i}
    \\ & = (\delta_2+1)q^{\delta}-\delta_2q^{\delta-1}.
\end{align*}
Substituting these formulas to~\eqref{eqn:efixedmonic} completes the proof of Theorem~\ref{thm:mainfixed}.

\section{Proof of Theorem~\ref{thm:main1f} and Corollary~\ref{cor:main1}} As in the proof of Theorem~\ref{thm:mainfixed}, we easily obtain 
\[
\denergy^{f} (\F_q[T]_{\delta}) = (q-1)^3 \denergy^{f} (\cM_{\delta}).
\] We now proceed with only considering monic polynomials.
The case $\delta=0$ easily follows by using the same argument as in the beginning of the proof of Theorem~\ref{thm:mainfixed}. Next, when $\delta>0$, we  consider all solutions to~\eqref{eqn:multen} with $\deg f_i=\delta_i$ for $i=1,2,3,4$, with $\delta_1+\delta_2=\delta_3+\delta_4=\delta$. We then obtain
\begin{align}\begin{split}\label{eqn:Ef1-min}
\denergy^{f} (\cM_{\delta}) &= \sum  \energy^{\cF} (\cM;\delta_1,\delta_2,\delta_3,\delta_4) \\
&= \sum [\min(\delta_1,\delta_2,\delta_3,\delta_4)+1]q^{\delta} - \min(\delta_1,\delta_2,\delta_3,\delta_4)q^{\delta-1},
\end{split}
\end{align} where the sum is taken over all nonnegative integer tuples $(\delta_1,\delta_2,\delta_3,\delta_4)$ with $\delta_1+\delta_2=\delta_3+\delta_4=\delta>0$.

We now rewrite the sum in~\eqref{eqn:Ef1-min} with respect to $\min(\delta_1,\delta_2,\delta_3,\delta_4)$ as follows:
\begin{equation}\label{eqn:Ef1-sum}
    \denergy^{f} (\cM_{\delta}) = \sum_{i=0}^{\delta} g(i)[(i+1)q^{\delta}-iq^{\delta-1}],
\end{equation}
where $g(i)$ is the number of nonnegative integer tuples $(\delta_1,\delta_2,\delta_3,\delta_4)$  satisfying $\delta_1+\delta_2=\delta_3+\delta_4=\delta$ and $\min(\delta_1,\delta_2,\delta_3,\delta_4)=i$. 

By substituting $\delta_j'=\delta_j'-i$ for $j=1,2,3,4$, the number of such tuples is the number of nonnegative integer tuples $(\delta_1',\delta_2',\delta_3',\delta_4')$ with \[
\delta_1'+\delta_2'=\delta_3'+\delta_4'=\delta-2i,
\] such that at least one of the variables is equal to zero. By considering the zero placements and applying principles of inclusion-exclusion, we obtain \begin{equation}
    g(i) =\begin{cases} 4(\delta-2i), \quad &\text{if} \quad \delta>2i,\\
    1, \quad &\text{if} \quad \delta=2i,\\
    0, \quad &\text{otherwise}.
    \end{cases}
\end{equation}

Substituting this to~\eqref{eqn:Ef1-sum}, we obtain
\begin{equation}\label{eqn:Ef1-parity}
    \denergy^{f} (\cM_{\delta}) =  \begin{cases} \displaystyle\sum_{i=0}^{(\delta-1)/2} 4(\delta-2i)[(i+1)q^{\delta}-iq^{\delta-1}]  &\text{when }\delta \text{ is odd,}\\
   \dfrac{\delta+2}{2}q^\delta - \dfrac{\delta}{2}q^{\delta-1}+ \displaystyle\sum_{i=0}^{(\delta-2)/2} 4(\delta-2i)[(i+1)q^{\delta}-iq^{\delta-1}]   &\text{when }\delta \text{ is even.}
        
    \end{cases}
\end{equation}
We then simplify~\eqref{eqn:Ef1-parity} in each case. Note that for all non-negative integer $A$, we have \begin{align*}\begin{split}
\sum_{i=0}^A 4(\delta-2i)(i+1) &=4 \delta \sum_{i=0}^A  (i+1) - 8\sum_{i=0}^A i^2 - 8\sum_{i=0}^A i \\&= 2\delta(A+1)(A+2) - \dfrac{4A(A+1)(2A+1)}{3} - 4A(A+1)\\
&= \dfrac{A+1}{3} \left(6\delta(A+2) - 4A(2A+1) -12A  \right)
\\&= \dfrac{2(A+1)(A+2)(3\delta -4A)}{3},
\end{split}
\end{align*}

\begin{align*}\begin{split}
&\sum_{i=0}^A 4i(\delta-2i) = -\dfrac{4A(A+1)(2A+1)}{3} + 2\delta A(A+1)= \dfrac{2A(A+1)(3\delta-4A-2)}{3}.
\end{split}
\end{align*} Substituting these to~\eqref{eqn:Ef1-parity}, we first obtain that when $\delta$ is odd,
\begin{align*}
  \denergy^{f} (\cM_{\delta}) & = \dfrac{2\delta(\delta+2)(\delta+4)}{12}  q^{\delta} - \dfrac{2(\delta-1)(\delta+1)\delta}{6}q^{\delta-1}\\&= \binom{\delta+3}{3} q^{\delta} - \binom{\delta+1}{3}q^{\delta-1}.
\end{align*}
When $\delta$ is even,
\begin{align*}
     \denergy^{f} (\cM_{\delta}) & = \dfrac{\delta+2}{2}q^\delta - \dfrac{\delta}{2}q^{\delta-1}+ \dfrac{(\delta+2)(\delta+4)\delta}{6} q^{\delta} - \dfrac{(\delta-2)\delta(\delta+2)}{6}q^{\delta-1}
     \\ & = \dfrac{(\delta+2)(\delta^2+4\delta+3)}{6}q^{\delta} -\dfrac{\delta(\delta^2-1)}{6}q^{\delta-1} 
     \\ &= \binom{\delta+3}{3} q^{\delta} - \binom{\delta+1}{3}q^{\delta-1},
\end{align*}
which therefore completes the proof of Theorem~\ref{thm:main1f}. 
Then, to prove Corollary~\ref{cor:main1}, we note 
\begin{align*}
  \denergy^{f} (\cM_{\le \delta}) &= \sum_{i=0}^\delta \denergy^{f} (\cM_{i})\\
 & =   1 +  \sum_{i=1}^\delta \left[\binom{i+3}{3} q^{i} - \binom{i+1}{3}q^{i-1}\right]
     \\ & =  \binom{\delta+3}{3}q^{\delta} +1+  \sum_{i=1}^{\delta-1}\left[ \binom{i+3}{3} -\binom{i+2}{3} \right]q^i
     \\ & =   \binom{\delta+3}{3}q^{\delta}  +  \sum_{i=0}^{\delta-1}\binom{i+2}{2}q^i.
\end{align*}

\section{Proof of Theorems~\ref{thm:detgen}--\ref{thm:detgendiv}}
 \subsection{Proof of Theorem~\ref{thm:detgen}} 
We now count $\aenergy(\cM_{\le\delta};g)$, the number of solutions to~\eqref{eqn:multg} under the constraints $\deg f_i\le \delta$. For monic polynomials $g_2\in \F_q[T]$, denote \[
\aenergy(\cM_{\le\delta};g;g_2) \coloneqq \#\{ (f_1,f_2,f_3,f_4)\in \cM_{\le \delta}^4 \colon f_1f_2-f_3f_4=g, \gcd(f_2,f_4)=g_2\}.
\]We only need to consider the case $g_2\mid g$, and we also have \begin{equation}\label{eqn:Aggsum}
\aenergy(\cM_{\le\delta};g) = \sum_{\substack{g_2\mid g}}\aenergy(\cM_{\le\delta};g;g_2).
\end{equation}

We first obtain the number of solutions to~\eqref{eqn:multg} with $\gcd(f_2,f_4)=1$, by combining Propositions~\ref{prp:73} and~\ref{prp:73g} over all possible tuples $(\delta_1,\delta_2,\delta_3,\delta_4)$.  We then obtain
\begin{equation}\label{eqn:Aggeneral}
    \aenergy( \cM_{\le \delta} ;g;1) = \sum_{n=\delta_g+1}^{2\delta} \left[\sum_{\substack{\delta_1+\delta_2=\delta_3+\delta_4=n\\ 1\le \delta_1,\delta_2,\delta_3,\delta_4\le \delta} 
    } (q^n-q^{n-1}) + \sum_{\substack{\delta_1+\delta_2=\delta_3+\delta_4=n\\ \min(\delta_1,\delta_2,\delta_3,\delta_4)=0}
    } q^n  \right] + O(q^{\delta_g-1})
\end{equation}

For the case where $g_2\ne 1$, we first let $(f_1^{(0)},f_2^{(0)},f_3^{(0)},f_4^{(0)})$ be a solution of~\eqref{eqn:multg} with $\gcd(f_2,f_4)=g_2$. Then, note that this tuple is a solution to~\eqref{eqn:multg} if and only if $(f_1^{(0)},f_2^{(0)}/g_2,f_3^{(0)},f_4^{(0)}/g_2)$ is a solution to \[
f_1f_2-f_3f_4=g/g_2 
\] with $\gcd(f_2,f_4)=1$ and $\deg f_1,\deg f_3 \le \delta$, $\deg f_2,\deg f_4 \le \delta-\deg g_2$. By this observation and applying~\eqref{eqn:Aggeneral}, we then obtain 
 
\begin{equation}\label{eqn:Aggcd}
    \aenergy( \cM_{\le \delta} ;g;g_2) = \sum_{n=\delta_g-\deg g_2+1}^{2\delta-\deg g_2} \left[\sum_{\substack{\delta_1+\delta_2=\delta_3+\delta_4=n\\ 1\le \delta_1,\delta_3\le \delta\\ 1\le \delta_2,\delta_4\le \delta-\deg g_2} 
    } (q^n-q^{n-1}) + \sum_{\substack{\delta_1+\delta_2=\delta_3+\delta_4=n\\ \min(\delta_1,\delta_2,\delta_3,\delta_4)=0\\ \delta_1,\delta_3 \le \delta\\\delta_2,\delta_4 \le \delta-\deg g_2}
    } q^n  \right] + O(q^{\delta_g-\deg g_2-1}).
\end{equation}
Note that~\eqref{eqn:Aggeneral} is a special case of~\eqref{eqn:Aggcd} when $g_2=1$.

We now discuss the asymptotic for $\aenergy( \cM_{\le \delta} ;g)$. First, excluding the error term, note that $\aenergy( \cM_{\le \delta} ;g;1,1)$ is a polynomial in $q$. Furthermore, this polynomial has degree $2\delta$, with leading coefficient coming only from the tuples $(\delta_1,\delta_2,\delta_3,\delta_4)=(\delta,\delta,\delta,\delta)$ in Proposition~\ref{prp:73}. Therefore, we  conclude that, as $q \to \infty$, \begin{equation}\label{eqn:Aggeneral-asymp}
       \aenergy( \cM_{\le \delta} ;g;1) = q^{2\delta}+ O(q^{2\delta-1}).
\end{equation}
For the other terms, we apply similar arguments to~\eqref{eqn:Aggcd} to obtain, as $q \to \infty$, \begin{equation}\label{eqn:Aggcd-asymp}
       \aenergy( \cM_{\le \delta} ;g;g_2) = q^{2\delta-\delta_{g_2}}+ O(q^{2\delta-\delta_{g_2}-1}).
\end{equation}
Substituting~\eqref{eqn:Aggeneral-asymp} and \eqref{eqn:Aggcd-asymp} to~\eqref{eqn:Aggsum} completes the proof of Theorem~\ref{thm:detgen}.

 \subsection{Proof of Theorem~\ref{thm:detconstant}}
We now proceed to derive an exact formula for $    \aenergy( \cM_{\le \delta} ;\Delta)$, for a  fixed $\Delta\in \F_q^*$. Following our argument for the proof of Theorem~\ref{thm:detgen}, we have \begin{align*}
    \aenergy( \cM_{\le \delta} ;\Delta) &= \sum_{n=1}^{2\delta} \left[\sum_{\substack{\delta_1+\delta_2=\delta_3+\delta_4=n\\ 1\le \delta_1,\delta_2,\delta_3,\delta_4\le \delta} 
    } (q^n-q^{n-1}) + \sum_{\substack{\delta_1+\delta_2=\delta_3+\delta_4=n\\ \min(\delta_1,\delta_2,\delta_3,\delta_4)=0}
    } q^n  \right].\end{align*}
Note that in this case, the error term $O(q^{\deg g-1})$ disappears since $\deg g =0$.

Separating the case $1\le n\le \delta$ and $\delta <n\le 2\delta$, we obtain when $\delta\ge 2$,
    \begin{align*}
     \aenergy( \cM_{\le \delta} ;\Delta)&= \sum_{n=1}^\delta  \left[(n-1)^2 (q^n-q^{n-1}) + 4nq^n \right] + \sum_{n=\delta+1}^{2\delta} \left[(2\delta+1-n)^2(q^n-q^{n-1})\right]
    \\ & = (\delta+1)^2 q^\delta + \sum_{n=1}^{\delta-1}\left[(n-1)^2 q^n - n^2 q^{n} + 4nq^n \right] + q^{2\delta} -\delta^2 q^{\delta}+ \sum_{n=\delta +1}^{2\delta-1} [(2\delta+1-n)^2 - (2\delta-n)^2] q^n
    \\ &= q^{2\delta} + (2\delta+1)q^{\delta}+\sum_{n=\delta+1}^{2\delta-1} (4\delta+1-2n)q^n  + \sum_{n=1}^{\delta-1} (2n+1)q^n
    \\ &= q^{2\delta}-1+ \sum_{n=0}^{\delta-1} (2\delta+1-2n)q^{\delta+n} + \sum_{n=0}^{\delta-1}(2n+1)q^n.
\end{align*}

Simplifying further, we obtain
\begin{align*}
     \aenergy( \cM_{\le \delta} ;\Delta)
     &=q^{2\delta}-1+ [(2\delta+1)q^{\delta}+1] \sum_{n=0}^{\delta-1} q^{n} -(2q^{\delta}-2) \sum_{n=0}^{\delta-1} nq^n
     \\ & = q^{2\delta} -1 + \dfrac{[(2\delta+1)q^{\delta}+1] (q^{\delta}-1)}{q-1} -\dfrac{q(2q^{\delta}-2)[(\delta-1)q^{\delta}-\delta q^{\delta-1}+1]}{(q-1)^2}
     \\ & = \dfrac{(q^{2\delta} -1 )(q-1)^2 + (q-1)(q^{\delta}-1) [(2\delta+1)q^{\delta}+1] -q(2q^{\delta}-2)[(\delta-1)q^{\delta}-\delta q^{\delta-1}+1]   }{(q-1)^2}
     \\& = \dfrac{q(q^{\delta}-1)(q^{\delta+1}+q^{\delta}+q-3)}{(q-1)^2}.
\end{align*}
which completes the proof.

 \subsection{Proof of Theorem~\ref{thm:detgendiv}}
To count $\aenergy^f(\cM_{\le\delta};g)$, we first denote 
\begin{align*}
&\aenergy^f( \cM_{\le \delta} ;g;g_2) \coloneqq \#\{ (f_1,f_2,f_3,f_4)\in \cM^4\colon f_1f_2-f_3f_4=g, f_1f_2,f_3f_4\in \cM_{\le \delta}, \gcd(f_2,f_4)=g_2 \}.
\end{align*} We proceed similarly as in the proof of Theorem~\ref{thm:detgen} and obtain
\begin{equation}\label{eqn:Aggsumdiv}
\aenergy^f(\cM_{\le\delta};g) = \sum_{\substack{g_2\mid g}}\aenergy^f(\cM_{\le\delta};g;g_2).
\end{equation}

Using Proposition~\ref{prp:73} with different summations, we obtain, as in the proof of Theorem~\ref{thm:detgen}, \begin{align*}
    \aenergy^f( \cM_{\le \delta} ;g;g_2) &= \sum_{n=\delta_g-\deg g_2+1}^{\delta-\deg g_2} \left[\sum_{\substack{\delta_1+\delta_2=\delta_3+\delta_4=n\\ \delta_1,\delta_2,\delta_3,\delta_4\ge 1} 
    } (q^n-q^{n-1}) + \sum_{\substack{\delta_1+\delta_2=\delta_3+\delta_4=n\\ \min(\delta_1,\delta_2,\delta_3,\delta_4)=0}
    } q^n  \right] + O(q^{\delta_g-\deg g_2-1})
    \\
    &=\sum_{n=\delta_g-\deg g_2+1}^{\delta-\deg g_2}  \left[(n+1)^2q^n-(n-1)^2q^{n-1}  \right]+ O(q^{\delta_g-\deg g_2-1}),
\end{align*} which alongside~\eqref{eqn:Aggsumdiv} implies that $ \aenergy( \cM_{\le \delta} ;g)$ is a polynomial with respect to $q$. In particular, we explicitly obtain \begin{equation}\label{eqn:Aggcddiv}
    \aenergy^f(\cM_{\le\delta};g) = \sum_{\substack{g_2\mid g}} \sum_{n=\delta_g-\deg g_2}^{\delta-\deg g_2} \left[(n+1)^2q^n-(n-1)^2q^{n-1}   \right] + O(q^{\delta_g-1+o(1)}).
\end{equation}

To obtain the leading term of~ $\aenergy^f(\cM_{\le\delta};g)$ (viewed as polynomials on $q$), we first note that the degree of $\aenergy^f(\cM_{\le\delta};g)$, viewed as a polynomial in terms of $q$, is $\delta$. Furthermore, the $q^\delta$ term only comes from $ \aenergy( \cM_{\le \delta} ;g;1)$. Viewing the summation, in~\eqref{eqn:Aggcddiv} we see that the number of $q^{\delta}$ terms is exactly the number of nonnegative integer solutions to $\delta_1+\delta_2=\delta_3+\delta_4=\delta$, which is $(\delta+1)^2$. Therefore, the leading term of $ \aenergy( \cM_{\le \delta} ;g)$ is exactly $(\delta+1)^2q^{\delta}$, which completes our proof.

\section*{Acknowledgements}
The author would like to thank anonymous reviewers for comments, corrections and suggestions.

\end{document}